\documentclass[12pt]{amsart} 

\usepackage{amssymb,epsfig}
\usepackage{stmaryrd}
\usepackage{amsmath,verbatim}
\usepackage[all,2cell,dvips]{xy}

\input{epsf.sty}

\include{diagram} 
\newtheorem{thmA}{Theorem}

\newtheorem{corA}{Corollary}

\newtheorem{theorem}{Theorem}[section]
\newtheorem{thm}[theorem]{Theorem}
\newtheorem{lemma}[theorem]{Lemma}

\newtheorem{cor}[theorem]{Corollary}

\newtheorem{prop}[theorem]{Proposition}

\theoremstyle{definition}
\newtheorem{defn}[theorem]{Definition}

\theoremstyle{remark}
\newtheorem{remark}[theorem]{Remark}

\numberwithin{equation}{section}

 \def\cat{{\rm{CAT}}$(0)$ }

\def\n{\hbox{\bf{n}}} 
\def\T{\mathcal T}
\def\ssm{\smallsetminus} 

\def\R{\mathbb R} 
\def\E{\mathbb E}
\def\C{\mathbb C}
\def\Si{\Sigma}
\def\L{{\rm{Lick}}}

\def\S{\Sigma}

\def\modsg{{\rm{Mod}}(\Sigma_g)}
\def\spg{{\rm{Sp}}(2g,\Z)}
\def\Mod{{\rm{Mod}}}

\def\fix{\hbox{{\rm{Fix}}}}

\def\L{\Lambda}

\def\<{\langle}
\def\>{\rangle} 

\def\-{\underline}

\def\N{\mathcal N} 
\def\F{\mathcal F} 
\def\Z{\mathbb Z} 
 
\def\G{\Gamma}

\def\a{\alpha}
\def\b{\beta}
\def\g{\gamma}

\def\-{\overline}

\def\FR{\text{\rm{F}}\mathbb R}
\def\isom{\text{\rm{Isom}}} 
\def\ball{\text{\rm{Ball}}}

\catcode`\@=11
\def\serieslogo@{\relax}
\def\@setcopyright{\relax}
\catcode`\@=12

\begin{document}

\title[On the dimension of  CAT$(0)$ spaces where mapping class groups act]
{On the dimension of  CAT$(0)$ spaces where mapping class groups act}
\author[Martin R. Bridson]{Martin R.~Bridson}
\address{Mathematical Institute,
24-29 St Giles',
Oxford OX1 3LB, U.K. }
\email{bridson@maths.ox.ac.uk} 


\subjclass{20F67, 57M50}

\keywords{mapping class groups, {\rm{CAT}}$(0)$ spaces, dimension,
fixed-point theorems}

\thanks{This research was supported by a Senior
Fellowship from the EPSRC and a Royal
Society Wolfson Research Merit Award.}

\begin{abstract} 
Whenever the mapping class group of a
closed orientable surface of genus $g$ acts by semisimple isometries
on a complete {\rm{CAT}}$(0)$ space of dimension less than $g$ it fixes a point.
\end{abstract}

\maketitle
  
\section{Introduction}  There are at least two interesting actions of
the mapping class group $\modsg$ of a closed orientable surface $\S_g$   
on a complete \cat space. The most classical of these actions comes from the action of 
$\modsg$ on the first homology of $\S_g$: the resulting epimorphism $\modsg\to \spg$ 
gives an action  of $\modsg$ on the Siegel upper half space
${\rm{\bf H}}_g$, which is
the symmetric space for the symplectic group ${\rm{Sp}}(2g,\R)$.
In this action the Torelli group acts trivially
while the Dehn twists in non-separating curves act as neutral parabolics. 
 
The natural 
action of $\modsg$ on its Teichm\"uller space $\T_g$ gives rise to the second
action. 
$\modsg$ preserves the Weil-Petersson metric $d_{\rm{WP}}$
on $\T_g$, which has non-positive curvature but is not complete \cite{wolp1,
wolp2}. 
The action of $\modsg$ on the 
completion of $(\T_g,d_{\rm{WP}})$  is faithful if
$g>2$ and is by semisimple
isometries; all of the Dehn twists act elliptically, i.e.~have fixed points.
(See \cite{mrb:bill, ursula} for more details.)  If $g\ge 3$ then $\modsg$ cannot act
properly by semisimple isometries on a complete \cat space \cite{KL}, \cite{BH} p.257.

The (real) dimension of ${\rm{\bf H}}_g$ is $g(2g+1)$
 while the dimension of $\T_g$ is $6g-6$.
I do not know if $\modsg$ can act by semisimple 
isometries on a complete {\rm{CAT}}$(0)$ space whose dimension is less than $6g-6$
without fixing a point. 
The main purpose of this note is to prove the following weaker bound, which was announced
in \cite{mrb:bill}. 

\begin{thmA} \label{t:main}
Whenever $\modsg$  acts by semisimple
isometries on a complete {\rm{CAT}}$(0)$ space of dimension less than $g$,
it must fix a point.
\end{thmA} 

If $X$ is a polyhedral space with only finitely many isometry types of cells, then
all cellular
isometries of $X$ are semisimple \cite{mrb:ss}. Thus Theorem \ref{t:main}
implies, in particular, that $\modsg$ cannot act without a global fixed
point on a Bruhat-Tits building of affine type whose rank is less than $ g$. (In the case
of trees, this was already known by Culler and Vogtmann \cite{CV}.)
Actions on affine buildings arise from linear
representations in a manner that we recall in Section \ref{s:reps}.
 Thus,  as Farb points out in \cite{farb} p.44, control over such actions leads to
constraints on the low-dimensional representation theory of $\modsg$.

\begin{corA}\label{c:evalues} Let  $K$ be a
field and let $\rho:\modsg\to {\rm{GL}}(g,K)$ be a representation.
If $K$ has characteristic $0$ then the eigenvalues of each $\rho(\gamma)\ (\gamma\in\G)$ 
are algebraic integers, and if $K$ has positive characteristic they are roots of
unity.
\end{corA}

\begin{corA}\label{c:reps} If $K$ is an algebraically closed
field and $d\le g$, then there are only finitely many conjugacy classes
of irreducible representations  $\modsg\to {\rm{GL}}(d,K)$.
\end{corA}
  
The value of these corollaries should be weighed against the
fact that no representations with infinite image are known to exist in such low degrees.
Indeed it seems plausible that there are no such representations below dimension $2g$.
Using our results, Louis Funar \cite{funar} recently proved that every homomorphism
$\modsg\to{\rm{SL}}(\lfloor\sqrt{g+1}\rfloor, \C)$ has finite image.

In the proof of Theorem \ref{t:main}, the semisimplicity hypothesis
is used to force the Dehn twists to have fixed points.  
 There are two useful
ways in which the semsimplicity assumption can be weakened. First,  it is enough to assume 
only that the action has
 no neutral parabolics (see Section \ref{s:last}). Alternatively, one can
simply assume  that the Dehn
twist in a certain type of separating loop has a fixed point (Theorem 5.2). By exploiting
this second observation we prove the following theorem.

Recall that $\gamma_3(G)$, the third term of  lower central series
of a group $G$ is defined by setting  $\gamma_2(G)=[G,G]$ and
$\gamma_3(G)=[G,\gamma_2(G)]$. So $G/\gamma_3(G)$
is the freeest possible 2-step nilpotent quotient of $G$, and
there
is a natural map ${\rm{Out}}(G)\to {\rm{Out}}(G/\gamma_3(G))$.
In the case $G=\pi_1\S_g$, Morita \cite{morita} (cf.~\cite{jo})
 has identified the
image of the map $\modsg \cong {\rm{Out}}_+(G)\to {\rm{Out}}(G/\gamma_3(G))$; it is isomorphic to
an explicit  semidirect product $A\rtimes {\rm{Sp}}(2g,\Z)$ with $A$ abelian. 

\begin{thmA}\label{t:K_g}
The image of $\modsg \to {\rm{Out}}(\pi_1\S_g/\gamma_3(\pi_1\S_g))$ fixes a point whenever it acts by isometries on a
complete {\rm{CAT}}$(0)$ space of dimension less than $g$.
\end{thmA}

This paper is organised as follows. In Section \ref{s:1} we gather such
facts as we need concerning {\rm{CAT}}$(0)$ spaces and their isometries.
Section \ref{s:2} contains a variation on the {\em{ample
duplicaton criterion}} from \cite{mrb:helly}. This is a condition on generating
sets that leads to the existence of fixed points for group
actions via a bootstrap
procedure; it is a refinement of the 
{\em{Helly technique}}, cf.~\cite{farb:helly}. In 
Section \ref{s:lick}
 we establish the facts that we need in order to show that
the Lickorish generators for $\modsg$ satisfy this criterion.
Section \ref{s:proofs} contains the proofs of Theorems \ref{t:main}
and \ref{t:K_g}. Section \ref{s:reps} contains a discussion of Corollaries \ref{c:evalues}
and \ref{c:reps}.

I thank Tara Brendle for drawing the figure in Section \ref{s:lick} and Benson Farb for his comments
concerning Section \ref{s:reps}.

\section{Isometries of {\rm{CAT}}$(0)$ spaces}\label{s:1}

 Let $X$ be a  geodesic metric space.
A geodesic triangle $\Delta$ in $X$ consists of three points $a,b,c\in X$ and 
three geodesics $[a, b],\, [b,c],\, [c,a]$. Let $\-\Delta\subset\E^2$ be a triangle
in the Euclidean plane with the same edge lengths as $\Delta$
and let $\overline x\mapsto
x$ denote the map  $\-\Delta\to\Delta$
that sends each side of $\-\Delta$ isometrically
onto the corresponding side of $\Delta$. 
One says that {\em{$X$ is a 
{\rm{CAT}}$(0)$ space}} if for all $\Delta$ and all $\-x,\-y\in\-\Delta$ the inequality
 $d_X(x,y)\le d_{\E^2}(\-x, \-y)$ holds.

We refer to Bridson and Haefliger
\cite{BH} for basic facts about {\rm{CAT}}$(0)$ spaces.
 
In a {\rm{CAT}}$(0)$ space there is a unique geodesic $[x,y]$ joining
each pair of points $x,y\in X$.
A subspace $Y\subset X$ is said to be {\em{convex}} if 
$[y,y']\subset Y$ whenever $y,y'\in Y$.  

We write $\isom(X)$ for the group of isometries of a metric space $X$
and $\ball_r(x)$ for the closed ball of radius $r>0$ about $x\in X$.
Given a subset $H\subset\isom(X)$, we denote its
set of common  fixed-points
$$\fix(H):=\{x\in X\mid \forall h\in H,\  h.x=x\}.$$
Note that if $X$ is {\rm{CAT$(0)$}} then $\fix(H)$  is closed and convex.

The isometries of a {\rm{CAT}}$(0)$ space $X$ divide naturally into
two classes: the {\em{semisimple}} isometries are those for which there exists $x_0\in X$ such that
$d(\gamma.x_0, \, x_0) = |\g|$ where $|\g|:=\inf\{d(\gamma.y, \, y) \mid y\in X\}$; the remaining
isometries are said to be
{\em{parabolic}}. A parabolic $\g$ is {\em{neutral}} if $|\g|=0$.
 Semisimple isometries are divided into {\em{hyperbolics}}, for which $|\g|>0$,
and {\em{elliptics}}, which have fixed points.   
If $\gamma$  is hyperbolic then there exist $\gamma$-invariant isometric copies of $\R$
in $X$ on which $\gamma$ acts as a translation by $|\gamma|$. 
Each such subspace is called an axis for $\gamma$.

\begin{prop}\label{l:hyp} Let $\G$ be a group acting by
isometries on a complete \cat\ space $X$. If $\gamma\in\G$
acts as a hyperbolic isometry then $\gamma$ has infinite order in the abelianisation of
its centralizer $Z_\G(\gamma)$.
\end{prop} 

\begin{proof} This is proved on page 234 of \cite{BH}.
The main points are these:  the union of the axes for $\gamma$ splits isometrically
as $Y\times\R\subset X$; this subspace
is preserved
by $Z_\G(\gamma)$, as is the splitting; the action on the
second factor gives a homomorphism from $Z_\G(\g)$ to the
abelian group ${\rm{Isom}}_+(\R)$ and  the image of $\g$
is non-trivial. 
\end{proof}

\begin{remark} \label{r:km}
A more subtle argument exploiting \cite{KM}
shows that non-neutral parabolics
have infinite image in the abelianisation of their centralizer. This
is explained  in the proof
of   \cite{mrb:bill} Theorem 1.
\end{remark}
   
The facts that we need about elliptic isometries rely on the following well-known
proposition (\cite{BH}, II.2.7).

Given a subspace $Y\subseteq X$, let $\rho(Y):=\inf\{r\mid 
Y\subseteq \ball_r(x), \text{ some }x\in X\}$.

\begin{prop}
If $X$ is a complete \cat space and $Y$ is a non-empty bounded
subset, then there is a unique point $c_Y\in X$ such that
$Y\subseteq \ball_{\rho(Y)}(c_Y)$.
\end{prop}

\begin{cor} Let $X$ be a complete \cat space. If $H<\isom(X)$ has
a bounded orbit then $H$ has a fixed point.
\end{cor}

\begin{proof} The centre $c_O$ of  any $H$-orbit $O$ will be a fixed
point.
\end{proof}

\begin{cor}\label{c:commutes}
 Let $X$ be a complete  {\rm{CAT}}$(0)$ space. If
the subgroups $H_1,\dots, H_\ell<\isom(X)$ commute and
$\fix(H_i)$ is non-empty for $i=1,\dots,\ell$, then 
$\bigcap_{i=1}^\ell \fix(H_i)$ is non-empty.
\end{cor}

\begin{proof} A simple induction reduces us to the case
$\ell=2$. Since $\fix(H_2)$ is non-empty, each $H_2$-orbit
is bounded. As $H_1$ and $H_2$ commute, $\fix(H_1)$ is
$H_2$-invariant and therefore contains an $H_2$-orbit.  Orthogonal
projection to a closed convex subspace, such as
$\fix(H_1)$, does not increase distances (\cite{BH}, II.2.4). Thus
 the centre of this $H_2$-orbit
is  in $\fix(H_1)$, providing us with a fixed point for $H_1\cup H_2$.
\end{proof}

\section{Nerves, fixed points, and bootstraps}\label{s:2}

In this section we isolate from \cite{mrb:helly} the ideas
needed to prove the version of the {\em{ample duplication criterion}}
that we need. 

We need the
 following standard device for encoding the intersections of
families of subsets.

\begin{defn} Let $X$ be a set. The {\em{nerve}} $\N$ of a family   of
subsets $F_\lambda \subseteq X\ (\lambda\in\Lambda)$ is the abstract simplicial complex with
vertex set $\Lambda$ that has a $k$-simplex with vertices
$\{\lambda_0,\dots,\lambda_k\}$ if and only if $\bigcap_{i=0}^kF_{\lambda_i}
\neq \emptyset$.

If the index set $\Lambda$
is $\n=\{0,\dots,n\}$, where $n\in\mathbb N$, then we
regard $\N$ as a subcomplex of the standard $n$-simplex $\Delta_n$.

We write $|\N|$ to denote the geometric realisation 
of $\N$.
\end{defn}

Helly's classical theorem    concerning the  intersection
of 
convex subsets in $\R^n$   can be formulated as follows:  {\em{Let $\{C_0,\dots,C_N\}$ be
a family of closed convex sets in $\R^n$. If the nerve $\N$ of this family
contains the full $n$-skeleton of $\Delta_N$, then $\N=\Delta_N$.}}
There are many variations on this theorem in the literature. The
following  special case of the one  proved in \cite{mrb:helly}
will suffice for our purposes.

\begin{thm}\label{t:helly} If $X$ is a complete
\cat space of dimension at most 
$d$ and $\N$ is the
nerve of a family of closed convex subsets of $X$, then
every continuous map $|\N|\to \mathbb S^d$ is homotopic to a 
constant map.
\end{thm}
 
Our proof of Theorem \ref{t:main} exemplifies the fact that 
by applying versions of Helly's theorem
to fixed point sets of subgroups, one 
can sometimes prove fixed point theorems for groups of geometric
interest  (cf.~\cite{L}, \cite{gangof6}).
Forms of this idea appear at various places in the literature.
 I learned it from Benson Farb \cite{farb:helly}.

\subsection{Joins}
Let $K_1$ and $K_2$ be abstract
simplicial complexes
whose vertex sets $V_1$ and $V_2$ are disjoint.  The {\em{join}} of $K_1$ and $K_2$, denoted
$K_1\ast K_2$ is a simplicial
complex with vertex set  $V_1\sqcup V_2$;
a subset of $V_1\sqcup V_2$ is
a simplex\footnote{i.e. is the vertex set of a simplex in the geometric
realisation}  of
$K_1\ast K_2$ if and only if it is a simplex of
$K_1$, a simplex of $K_2$, or the union of a
simplex of $K_1$ and a simplex of $K_2$.
For example, the join of an $n$-simplex and an
$m$-simplex is an $(n+m+1)$-simplex.  Note that
$$
\dim (K_1\ast K_2) = \dim K_1 + \dim K_2 +1,
$$
provided $V_1$ and $V_2$ are non-empty. If the geometric realisation of $K_i$ is a sphere
of dimension $d_i$, for $i=1,2$, then the geometric realization of $K_1\ast K_2$ is a sphere of
dimension $d_1+d_2+1$; so $\partial\Delta_n\ast\partial\Delta_m$ is a triangulation of
$\S^{n+m-1}$.
 
\begin{prop}\label{c:join} Let $X$ be a complete {\rm{CAT}}$(0)$
space and let $S_1,\dots,S_\ell\subseteq \isom(X)$ be subsets
such that $[s_i,s_j]=1$ for all $s_i\in S_i, \, s_j\in S_j\ (i\neq j)$.
If $\N_i$ is the nerve of the family $\F_i=\{\fix(s_i)\mid s_i\in S_i\}$,
then the nerve $\N$ of $\F_1\cup\dots\cup\F_\ell$ is
$\N_1\ast\dots\ast\N_\ell$.
\end{prop}

\begin{proof} It is clear that $\N$ is contained in the join of
the $\N_i$; we must argue  that the converse is true, i.e.~that  
$\N$ has as a simplex  the union of  each
$\ell$-tuple of simplices  $\sigma_i\subset\N_i\ (i=1,\dots,\ell)$.
The $\sigma_i$
correspond to an $\ell$-tuple of commuting subgroups $H_i$ of 
$\isom(X)$, namely the subgroups generated by the elements  of
$S_i$ indexing the vertices of $\sigma_i$. The presence of $\sigma_i$
in $\N_i$ is equivalent to the statement that $\fix(H_i)$ is non-empty.
Corollary \ref{c:commutes}
then tells us $\fix(\cup_i H_i)=\cap_i\fix(H_i)$ is non-empty, as required.
\end{proof}

\begin{prop}[Bootstrap Lemma]\label{l:bootstrap}
Let $k_1,\dots,k_n$ be positive
integers and let $X$ be a complete \cat space of dimension less than $k_1+\dots+k_n$.
Let $S_1,\dots,S_n \subset\isom(X)$ be subsets with $[s_i,s_j]=1$
for all $s_i\in S_i$ and $s_j\in S_j\ (i\neq j)$.

If, for $i=1,\dots,n$, each $k_i$-element subset of $S_i$ has a 
fixed point in $X$, then for some $i$ every 
finite subset of   $S_i$ has a fixed point.
\end{prop} 

\begin{proof} Suppose that the conclusion
of the proposition were false.
Then for $i=1,\dots,n$ there would be a smallest
integer $k_i'\ge k_i$ such that some
$(k_i'+1)$-element subset 
$T_i=\{s_{i,1},\dots, s_{i,k_i'+1}\}$ in $S_i$  did not have a common fixed point.  

Since any $k_i'$
elements of $T_i$ have a common fixed point, 
the nerve of the family $\F_i=\{\fix(s_{i,1}),\dots,\fix(s_{i,k_i'+1})\}$
would be the boundary of a $k_i'$-simplex $\partial\Delta_{k_i'}$. Hence, by
Proposition \ref{c:join}, the nerve of  $\F_1\cup\dots\cup \F_n$  would be the join  $\partial\Delta_{k_1'}\ast
\dots\ast\partial\Delta_{k_n'}$. But this
contradicts Theorem \ref{t:helly}, because 
the realisation of this join
is homeomorphic to a sphere of dimension $(\sum_{i=1}^n k_i') -1\ge\dim X$.
\end{proof}

Since isometries of finite order have fixed points, 
taking $k_i=1$ we get:

\begin{cor}\label{c:torsion} If each of the
groups $\G_1,\dots,\G_n$ has a finite
generating set consisting of elements of 
finite order,
then at least one of the $\G_i$ has a fixed point whenever $D=\G_1\times\dots\times\G_n$ acts by isometries of a complete \cat space of dimension
less than $n$.
\end{cor}
 
When applying the above proposition one has to
overcome the fact that the conclusion only applies
to {\em{some}} $S_i$.
 A convenient way of gaining more control 
 is to restrict attention to conjugate sets.

\begin{cor}[Conjugate Bootstrap]\label{c:conjug}
Let $k$ and $n$ be positive
integers and let $X$ be a complete \cat space of dimension less than $nk$.
Let $S_1,\dots,S_n$ be conjugates of a subset  
$S\subset\isom(X)$ with $[s_i,s_j]=1$
for all $s_i\in S_i$ and $s_j\in S_j\ (i\neq j)$.

If each $k$-element subset of $S$ has a 
fixed point in $X$, then so does each finite subset
of $S$.
\end{cor}

\subsection{Some surface topology}

The reader will  recall that, given a 
closed orientable surface $\Si$ and two compact
homeomorphic
subsurfaces  $T,T'\subset \Si$,
there exists an automorphism of $\Si$ taking $T$
to $T'$ if and only
$\Si\ssm T$ and
$\Si\ssm T'$ are homeomorphic. In particular,
two homeomorphic subsurfaces are in the
same orbit under the action of ${\rm{Homeo}}(\Si)$ if the
complement of each is connected.

The relevance of this observation to our purposes
here is the following lemma, which will be used in tandem with the Conjugate Bootstrap.

\begin{lemma}\label{l:disjoint}  Let $H$ be the 
subgroup of $\Mod(\S)$ generated by the Dehn
twists in a set of loops  
 all of which are
contained in a compact subsurface $T\subset\Si$
with connected complement.
If $\Si$ contains $m$ mutually disjoint subsurfaces
$T_i$ homeomorphic to $T$, each with connected
complement, then $\Mod(\S)$ contains $m$
mutually-commuting conjugates $H_i$ of $H$. 
\end{lemma}

\begin{proof} Let $\phi_{i}$ be the mapping class of a 
homeomorphism $\Si\to\Si$ carrying
$T$ to $T_i$. The subgroups
$H_i:=\phi_i H\phi_i^{-1}$
 are supported in disjoint subsurfaces and therefore
commute.
\end{proof}

\section{The Lickorish generators}\label{s:lick}

Let $X$ be a complete \cat space of dimension less than $g$ and suppose
that $\modsg$ acts by isometries on $X$. We shall prove in Lemma 5.3  that if 
the Dehn twist  in each simple closed curve $C\subset\Sigma$ fixes a point of $X$,  
then the Dehn twists in any pair of curves with $|C\cap C'|=1$ have a common
fixed point. This provides the base step in an inductive argument that we use to prove that
increasingly large finite sets of Dehn twists have common fixed points. We achieve the inductive step
by applying the Conjugate Bootstrap Criterion (Corollary \ref{c:conjug}); this requires us to find
a certain number of commuting conjugates of whatever set of twists we are considering. We find
these commuting conjugates by applying Lemma \ref{l:disjoint}.

 For this strategy to work, we need a  set of Dehn twists that generates $\modsg$ and is such that one can obtain good estimates
 on the complexity of the minimal subsurface on which any subset of size $n$ is supported, as $n$ grows. The purpose 
of the present (rather technical) section is to obtain such estimates for the {\em{Lickorish generators}}.  (Working with
the smaller set of Humphries  generators  \cite{humph} would  reduce the number of cases that we have to  consider in 
our analysis, but the corresponding loss of symmetry would obscure some aspects of the argument.)

We fix a closed orientable surface $\S_g$ of genus $g\ge 2$.
Whenever we speak of a {\em{subsurface}} it is
to be understood that the subsurface is
compact and connected.

Building on classical work of Max Dehn \cite{dehn},
Raymond Lickorish \cite{lick}
proved that the mapping class
group of $\S_g$
is generated by the Dehn twists in $3g-1$
curves $\a_i,\b_i\, (1\le i\le g),\, \g_i\, (1\le i < g)$ portrayed in figure 1.
These are pairwise disjoint
except that  $|\a_i\cap\b_i| = |\b_i\cap\g_i|=|\b_i\cap\g_{i-1}| =1$.
Let $\L$ denote this set of curves.

We say that a subset
$S\subset \L$ is {\em{connected}} if the union 
$U(S)$ of the
loops in $S$ is connected. 

Let $A_{ij}=
\{\alpha_k \mid i\le k\le j\}$, let $B_{ij}=
\{\b_k \mid  i\le k\le j\}$, and let $C_{ij}=
\{\g_k \mid  i\le k\le j\}$. We
 distinguish the following connected subsets of
$\L$. 
\begin{equation}\label{intervals}
\begin{aligned}\\
\llbracket \b_i,\b_j\rrbracket &:=A_{i+1,j-1}\cup B_{ij}\cup C_{i,j-1}\\ 
\llbracket \b_i,\a_j\rrbracket  &:= \llbracket \b_i,\b_j\rrbracket \cup\{\a_j\}\  \
&\llbracket \a_i,\b_j\rrbracket  := \{\a_i\}\cup \llbracket \b_i,\b_j\rrbracket \\
\llbracket \b_i,\g_j\rrbracket &:=\llbracket \b_i,\b_j\rrbracket \cup \{\g_j\}\ \
 &\llbracket \g_i,\b_j\rrbracket  := \llbracket \b_i,\b_j\rrbracket \ssm\{\b_i\}\\
\llbracket \a_i,\a_j\rrbracket  &:=\{\a_i\}\cup \llbracket \b_i,\a_j\rrbracket \ \
&\llbracket \g_i,\g_j\rrbracket :=\llbracket \b_i,\g_j\rrbracket \ssm\{\b_i\}\\
\llbracket \g_i,\a_j\rrbracket &:=\llbracket \g_i,\b_j\rrbracket \cup\{\a_j\}\ \
&\! \llbracket \a_i,\g_j\rrbracket  :=\llbracket \a_i,\b_j\rrbracket \cup\{\g_j\}.
\end{aligned}
\end{equation}
 Roughly speaking, these are the
sets of loops in $\L$  that one encounters as one proceeds 
clockwise  in figure 1 from a loop with lesser index
$i$ to one of greater index $j$.

\subsection{Good neighbourhoods of subsets of
$\L$}

In order to describe ``good neighbourhoods" we need
 the following notation. First,
for $i<j$ we define
 $\delta_{ij}$ to be a loop that
  cobounds with $\alpha_i$
a subsurface $T(i,j)$
of genus $j-i$ with two boundary components that contains all of the loops $\llbracket \g_i,\b_j\rrbracket $; in the language of figure 1, it consists of one half of handle $i$ and all of the handles with index
 $i+1$ to $j$ inclusive.
The definition of $\delta_{ij}$ for $j<i$ is similar except that  the subsurface $T(i,j)$ 
that it cobounds with
$\alpha_i$ now consists of one half of handle $i$ and all of the handles with index from $j$ to
$i-1$ inclusive.

$\delta_{13}$ is shown in figure 1 and $\delta_{31}$
is obtained from $\delta_{13}$ by rotating the surface through
an angle $\pi$ about the axis of symmetry that
lies in the plane of the paper and intersects $\alpha_2$.

Finally, we need a loop $c_{ij}$ that separates
the handles  $i$ to $j$ from the
remainder of the surface and a loop $d_{ij}$
that, together with $\alpha_i$ and $\alpha_j$,
separates handles $(i+1)$ to $(j-1)$ from
the remainder of the surface; $c_{13}$ and $d_{13}$ are shown
in figure 1.

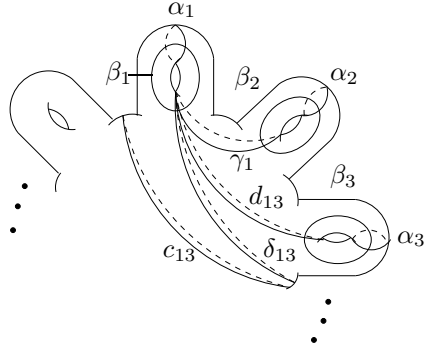
\begin{figure}
$$
\setlength{\unitlength}{0.05in}
\begin{picture}(0,0)(0,11)
\put(15.5,45.5){ {\bf {\footnotesize $\alpha_1$} } }
\put(9,38.5){ {\bf {\footnotesize $\beta_1$} } }
\put(22,29.5){ {\bf {\footnotesize $\gamma_1$} } }
\put(32.5,38.5){ {\bf {\footnotesize $\alpha_2$} } }
\put(22.5,38){ {\bf {\footnotesize $\beta_2$} } }
\put(24,25.5){ {\bf {\footnotesize $d_{13}$} } }
\put(39.5,21.5){ {\bf {\footnotesize $\alpha_3$} } }
\put(32.5,27.7){ {\bf {\footnotesize $\beta_3$} } }
\put(25.5,20){ {\bf {\footnotesize $\delta_{13}$} } }
\put(15,20){ {\bf {\footnotesize $c_{13}$} } }
\put(12.5,39){\line(1,0){2.4}}
\end{picture}
\includegraphics[width=2in]{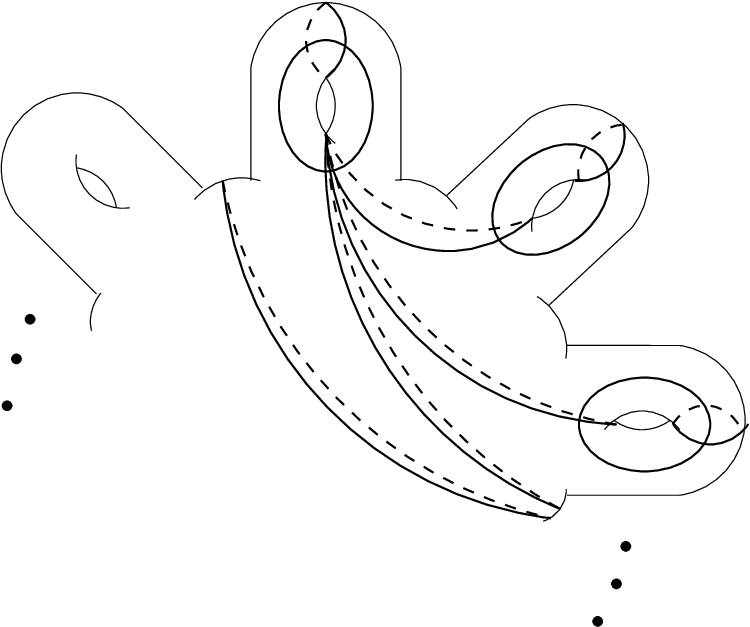}
$$
\caption{The loops used in the main argument}
\end{figure}
 
 Each of the following six lemmas is an
elementary observation, but
 their cumulative import is non-trivial.

\begin{lemma}\label{l:aa}
 The union of the loops in $\llbracket \alpha_i,\alpha_j\rrbracket \subset \L$ 
is contained in a  subsurface of genus $j-i+1$
with boundary  $c_{ij}$.
\end{lemma}

\begin{lemma}\label{l:gg} The union of the loops in
$\llbracket \gamma_i,\gamma_j\rrbracket \subset \L$
is contained in a subsurface of genus $j-i$
with boundary $\alpha_i\cup d_{i,j+1}\cup \a_{j+1}$.
\end{lemma}

\begin{lemma}\label{l:ag} The union of the loops in
$\llbracket \alpha_i,\gamma_j\rrbracket \subset \L$
is contained in the subsurface $T(j,i)$,
which has genus $(j-i+1)$ and boundary 
$\alpha_j\cup \delta_{ji}$. 
\end{lemma}

\begin{lemma}\label{l:ga} The union of the loops in
$\llbracket \gamma_i,\alpha_j\rrbracket \subset \L$
is contained in the subsurface $T(i,j)$,
which has genus $(j-i+1)$ and  boundary 
$\alpha_i\cup\delta_{ij}$. 
\end{lemma}

An {\em{$m$-chain}} in a surface $\Si_g$ is a finite sequence of simple closed
loops   $C_1,\dots, C_m$ such that 
$|C_i\cap C_{i+1}|=1$ and $C_i\cap C_j=
\emptyset$ if $|i-j|>1$. 

\begin{lemma}\label{l:goodChains}
 If
$m$ is even, then a closed
regular neighbourhood of an $m$-chain is a
surface of genus $m/2$ with one boundary
component. If $m$ is odd, then such a neighbourhood
is a  surface of
genus $(m-1)/2$ with two boundary
components.
\end{lemma}

In the odd case one distinguishes 
between {\em{non-separating}} and {\em{separating}} chains according to whether or
not the complement in $\Si_g$ of the union of the
loops in the chain is connected or not.

\begin{lemma}\label{l:badChains}
The only separating $m$-chains
in $\L$ are those that are obtained from
$\llbracket \a_i,\a_{j}\rrbracket $ by deleting the loops $\a_k$ with $i<k<j$,
where $j=i+\frac 1 2(m-3)$.
\end{lemma}

Recall that 
$S\subset \L$ is said to be
 {\em{connected}} if the union 
$U(S)$ of the
loops in $S$ is connected. 

\begin{prop}\label{p:size} Let $S\subset \L$ be a connected
subset.
\begin{enumerate}
\item
If $|S|=2\ell$ is even, then $U(S)$ is either contained
in a  subsurface of genus  $\ell$ with $1$ boundary
component, or else  in a non-separating subsurface
of genus at most $\ell-1$ with $3$ boundary components.

\item
If $|S|=2\ell+1$ is odd, then $U(S)$ is either contained
in a non-separating subsurface of genus $\ell$ with
at most $2$ boundary components, or else is contained
in a non-separating  subsurface of genus  
at most $\ell-1$ that has at most $3$ boundary components.
\end{enumerate}
\end{prop}

\begin{proof} If $S$ is an
$m$-chain  then
we can appeal to Lemma \ref{l:goodChains} except
in the case where $S$ is separating and $|S|=2\ell+1$ is odd. 
In this case $S\subset \llbracket \a_i,\a_j\rrbracket $, by Lemma \ref{l:badChains}, where 
$j=i+\ell-1$. In this case we appeal to Lemma \ref{l:aa}.

If $S$ is not itself a chain, then there is an $m$-chain
$C_1,\dots,C_m$
in $\L$, with $m$ strictly less than $|S|$, such that $S$
is contained in the set $I=\llbracket C_1,C_m\rrbracket \subset \L$, where
$\llbracket C_1,C_m\rrbracket $ is one of the ``interval like"
subsets described in (\ref{intervals}). We analyze each of the possibilities
using the subsurfaces given in Lemmas 
\ref{l:aa} to \ref{l:ga} to 
enclose $U(S)$ in a subsurface of controlled type.

\smallskip

First suppose that $m\le 2\ell -1$ is odd. 

If $I=\llbracket \a_i,\a_j\rrbracket $ then $j\le i+\ell -2$ and $I$ is contained in a
surface of genus at most $(\ell -1)$ with $1$ boundary component (Lemma \ref{l:aa}).

If $I=\llbracket \b_i,\b_j\rrbracket $ then we can enclose it in
$\llbracket \a_i,\a_j\rrbracket $, and $j\le i+\ell -1$. Thus $I$ is contained in a
surface of genus at most $\ell$ with $1$ boundary component  (Lemma \ref{l:aa}).

If $I=\llbracket \g_i,\g_j\rrbracket $ then $j\le i+\ell -1$ and $I$ is contained in a
surface of genus at most $(\ell-1)$ with $3$ boundary components  (Lemma \ref{l:gg}).

If $I=\llbracket \a_i,\g_j\rrbracket $ then $j\le i+\ell -2$ and $I$ is contained in a
surface of genus at most $(\ell-1)$ with $2$ boundary components  (Lemma \ref{l:ag}).

If $I=\llbracket \g_i,\a_j\rrbracket $ then $j\le i+\ell -1$ and $I$ is contained in a
surface of genus at most $(\ell-1)$ with $2$ boundary components.

\smallskip

Now suppose that $m=2 k$ is even.  

If $I=\llbracket \a_i,\b_j\rrbracket $ or $I=\llbracket \b_i,\a_j\rrbracket $ then
we can embed $I$ in $\llbracket \a_i,\a_j\rrbracket $
with $j\le i+k -1$, and this is contained in a
surface of genus at most $k$ with $1$ 
boundary component (Lemma \ref{l:aa}).

If $I=\llbracket \b_i,\g_j\rrbracket $ then we enclose it in $\llbracket \a_i,\g_j\rrbracket $ with
$j\le i+k -1$, which is contained in a
surface of genus at most $k$ with $2$ boundary components (Lemma \ref{l:ag}).

Finally, if 
$I=\llbracket \g_i,\b_j\rrbracket $ then we enclose it in $\llbracket \g_i,\a_j\rrbracket $ with
$j\le i+k$, which is contained in a
surface of genus at most $k$ with $2$ boundary components (Lemma \ref{l:ga}).

\smallskip

To complete the proof of (1) note that if $m=2k$ is even then
$k<\ell$, hence the required estimate.
In case (2) we have $k\le \ell$, so the above
estimates are again sufficient.
\end{proof}

 We write $\Sigma_{h,n}$ to denote the compact connected
 orientable surface of genus $h$ with $n>0$ boundary
 components.

\begin{lemma}\label{l:fit}$\ $
\begin{enumerate}
\item $\Si_{g}$ contains
$\lfloor g/\ell\rfloor$ disjoint subsurfaces homeomorphic  to 
$\Si_{\ell,1}$.
\item  $\Si_{g}$ contains
$\lfloor g/\ell\rfloor$ disjoint non-separating
subsurfaces homeomorphic to $\Si_{\ell-1,3}$.
\item  $\Si_{g}$ contains
$\lfloor (g-1)/\ell\rfloor$  disjoint 
non-separating
subsurfaces homeomorphic to $\Si_{\ell,2}$.
\end{enumerate}
\end{lemma}

\begin{proof} To prove (1), one expresses $\Si_g$
as the connected sum of $\mathbb S^2$ with $q$ disjoint
copies of $\Si_{\ell}$ and one copy
of $\Si_{\ell'}$, where $q=\lfloor g/\ell\rfloor$
and $\ell' = g-q\ell$.

\smallskip

For (3), note first that
given $q$ compact orientable surfaces 
$U_1,\dots,U_q$ of 
genus $h_1,\dots,h_q$, 
each with
$2$ boundary components $\partial U_i =
\partial U_i^+\cup\partial U_i^-$,
 one obtains a closed surface of genus
$1+\sum_i h_i$ by arranging the $U_i$ in
cyclic order and identifying $  \partial U_i^+$
with $\partial U_{i+1}^-$ (indices $\mod q$).

More generally, given an integer
$g\ge 1+\sum_i h_i$, one can 
obtain a closed surface  
of genus $g$ by
including into the above cyclic assembly
an additional surface of genus 
$g-1-\sum_i h_i$  with $2$ boundary
curves.  

Reversing perspective, we deduce from the preceding construction that given a closed surface $\Sigma$
of genus $g$ and  an integer $h$, one
can find $\lfloor \frac{g- 1} h\rfloor$ disjoint,
non-separating subsurfaces of genus $h$, each with
$2$ boundary components. This proves (3).

\smallskip

To prove (2), we assemble $\Si_g$ from
$q=\lfloor g/\ell\rfloor$ copies $F_i$ of $\Si_{\ell-1,3}$ and one copy of $\Si_{\ell',q}$, where
$\ell' = g-q\ell$: label the 3 boundary
components of $F_i$ as $\partial_i^+,\, \partial_i^0,\, \partial_i^-$ and identify $\partial_i^+$ with $\partial_{i+1}^-$
(indices $\mod q$) to obtain an orientable surface of
genus $q(\ell -1)+1$ with $q$ boundary components, then
cap-off the boundary by attaching  $\Si_{\ell',q}$.

\smallskip

In each of the above constructions, the relevant  subsurfaces that we described
are not quite disjoint --- they intersect along their boundaries ---
but this is remedied by an obvious retraction of each subsurface away from its boundary. \end{proof}

\subsection{A numerical lemma}

\begin{lemma}\label{l:count} Let $g\ge 1$ and $k\in [2, 2g]$ be integers.
If $k$ is even then $ (k-1) \lfloor 2g/k\rfloor \ge g$. If $k$ is
odd then $(k-1)\lfloor 2(g-1)/(k-1)\rfloor \ge g$.
 \end{lemma} 
 
 \begin{proof} Fix $k$ even. For every positive integer $s$,
 the function $\phi_k(x)= (k-1) \lfloor 2x/k\rfloor$ is constant on $[sk/2, (s+1)k/2)$, so if the inequality
 $\phi_k(g)  \ge g$ were to fail for some
  $g\ge k/2$ then it would fail
 at $g=\frac 1 2 (s+1)k -1$, with $s\ge 1$. But at this value
 $\phi_k(g) = (k-1)s$, and $(k-1)s < \frac 1 2 (s+1)k -1$
 implies $s(k-2) < k-2$, which is nonsense.
 
 The proof for $k$ odd is similar.
 \end{proof}
 
\section{Proof of the Main Theorems}\label{s:proofs}

We shall deduce Theorems \ref{t:main} and \ref{t:K_g} from
the following more technical result.

\begin{defn} A 
{\em{handle-separating loop}} on a closed orientable
surface $\S_g$ of genus $g\ge 2$ is a simple closed curve $c$
such that one of the components of $\S_g\ssm c$ has genus $1$.
\end{defn}

\begin{thm}\label{t:technical}
 If the mapping class group $\modsg$ of a
closed orientable surface of genus $g\ge 2$ acts by isometries
on a complete {\rm{CAT}}$(0)$ space $X$ of 
dimension less than $g$
and the Dehn twist in a handle-separating loop fixes a point
of $X$, then $\modsg$ fixes a point of $X$.
\end{thm}

\begin{lemma}\label{l:2loops}
Let $g\ge 2$ and consider an action of $\modsg$
by isometries on a complete
\cat space $X$ of dimension less than $g$.
Let $C_0$ be a handle-separating
loop on $\Si_g$.

If the Dehn twist about $C_0$ has a fixed point in $X$,
then so does the subgroup generated by the Dehn twists in any pair of
simple closed curves $C,C'$ on $\Si_g$ with $|C\cap C'|=1$.
\end{lemma}

\begin{proof} Each pair of
simple closed curves $C,C'$ on $\Si_g$ 
with $|C\cap C'|=1$ is contained in a subsurface 
$W$ of genus 1
with one boundary component. Thus it will be enough to show that
there is a point fixed by the subgroup $M_W
\subset\modsg$ consisting of the mapping classes
of homeomorphisms that stabilize $W$ and act
trivially on its complement.

As in the proof of Lemma \ref{l:fit}, we express
$\Si_g$ as the union of $g$ subsurfaces $W_i$ of
genus $1$, each with $1$ boundary component,
and a sphere with $g$ discs deleted. Let
$C_1,\dots,C_g$ be the boundary curves of
these subsurfaces. Each of the $C_i$ lies in the $\modsg$-orbit of $C_0$ and hence the Dehn twist
$T_{C_i}$ 
is an elliptic isometry of $X$. Since the loops $C_i$
are disjoint, these Dehn twists commute and hence
the set $F$ of points fixed by all of them is non-empty
(corollary \ref{c:commutes}).
 $F\subset X$ is closed and convex, and hence is itself
a complete \cat\ space of dimension less than $g$.

The centralizer of $\langle T_{C_1},\dots,T_{C_g}
\rangle$ in $\modsg$ preserves $F$. This
centralizer
 contains $D=M_1\times\dots\times M_g$, 
 where for brevity we have written
 $M_i$ in place of $M_{W_i}$. 
 Note that each $M_i$
 is conjugate to $M_W$.
 
 We are interested
 in the action of $D$ on $F\subset X$. Since the
 $T_{C_i}$ act trivially on $F$, this action factors
 through the product of the groups
   $M_i/\langle T_{C_i}\rangle$, each of which
     is isomorphic to the mapping class group
 of a genus $1$ surface with one puncture.
 This last group is isomorphic to
 ${\rm{SL}}(2,\Z)$, which is generated by a pair
 of torsion elements. Thus  Corollary \ref{c:torsion} 
 tells us that one of the $M_i/\langle T_{C_i}\rangle$
 has a fixed point in $F$. Hence $M_i$ has a fixed
 point in $F\subset X$. And since $M_W$ is conjugate to $M_i$
 in $\modsg$,
  it too has a fixed point in $X$.
 \end{proof}
 
\subsection{The proof of Theorem \ref{t:technical}}
A complete \cat\ space
of dimension $1$ is an $\mathbb R$-tree.
Culler and Vogtmann  \cite{CV} proved that if $g\ge 2$ then $\modsg$ has a fixed point whenever
it acts by isometries on an $\mathbb R$-tree.
Thus we may assume that $g\ge 3$.

Let $X$ be a complete \cat space of 
dimension less than $g$  on which $\modsg$ acts
by isometries so that the Dehn twist in a handle-separating
loop fixes a point of $X$.
We write $T(S)$ to denote the subgroup of $\modsg$ generated by the 
Dehn twists in a set of loops $S\subset\L$.
We must show that
$\modsg$ fixes a point of $X$. This is equivalent to showing
that $T(S)$ has a fixed point for every subset $S\subseteq\L$.

We proceed by induction, considering
a subset $S\subset\L$ such that $T(S')$ has a fixed point for
all  subsets $S'\subset\L$ of cardinality
less than $|S|$.  Lemma \ref{l:2loops}
covers the case $|S|\le 2$.

Suppose first that $S$ is not connected, say
$S=S_1\cup S_2$ with $U(S_1)\cap U(S_2)=\emptyset$.
The subgroups $T(S_1)$ and $T(S_2)$ 
commute, and each has a
fixed point since $|S_i|<|S|$, so
Corollary \ref{c:commutes} tells us that $T(S)$ 
has a fixed point. 

Suppose now that $S$ is connected. 
If $|S|=2\ell$ is even then Proposition \ref{p:size} tells us
that $U(S)$ is contained either in a subsurface of genus
$\ell$ with $1$ boundary component or else in a 
non-separating  subsurface of genus
$\ell-1$ with $3$ boundary components. Lemma \ref{l:fit}
tells us that in either case
 one can fit $\lfloor g/\ell\rfloor$ disjoint
copies of this subsurface into $\Si_g$.
Lemma \ref{l:disjoint} then provides us with $\lfloor g/\ell\rfloor$ 
mutually-commuting conjugates of $T(S)$. 
As all proper subsets of $S$ are assumed to have a fixed point,
the Conjugate Bootstrap (Corollary \ref{c:conjug}) tells us
that $T(S)$ will have a fixed point provided that the dimension of
$X$ is less than $(2\ell -1)\lfloor g/\ell\rfloor$. And since we are
assuming that $\dim X < g$, Lemma \ref{l:count} tells us
that this is the case.

The argument for $|S|$ odd is entirely similar. $\qed$

\subsection{The proof of Theorem \ref{t:main}}\label{s:last}

The following lemma explains why Theorem \ref{t:main} 
is a consequence of Theorem \ref{t:technical} and the fact that 
$\Mod(\Si_2)$ has property $\FR$, by \cite{CV}.

\begin{lemma}\label{l:1loop}
If $g\ge 3$, then whenever $\modsg$ acts by isometries on a 
complete \cat\ space,  the Dehn twist $T$ in
each simple closed curve  either fixes a point or acts as a neutral parabolic
 (i.e. $T$ is not {\em{ballistic}} in the terminology of \cite{CM}).
\end{lemma}

\begin{proof} The abelianisation of the centralizer of $T$ in $\modsg$ is finite (cf.~\cite{korkmaz}
and \cite{mrb:bill} Proposition 2),
so Proposition  \ref{l:hyp} and Remark \ref{r:km} imply that the translation number of $T$ is
zero.
\end{proof}

\subsection{The proof of Theorem \ref{t:K_g}}

The Dehn twists in separating curves lie in the kernel of
the natural map $\modsg\to {\rm{Out}}(\Sigma_g/\gamma_3(\Sigma_g))$.
(Johnson \cite{jo} proved that they generate the kernel
but we do not need that.) Thus Theorem \ref{t:K_g} is an immediate
consequence of Theorem \ref{t:technical}. $\qed$
 
\section{Constraints on the representation theory of $\modsg$}\label{s:reps}

In Section I.6.2  of \cite{serre} Serre proves that if a countable group $G$
has property FA and if $k$ is a field, then for every representation $\rho: G\to {\rm{GL}}(2, k)$
and every $g\in G$, the eigenvalues of $\rho(g)$ are integral over $\Z$, i.e.~each is the solution
of a monic polynomial with integer coefficients. For the convenience of the
reader we reproduce Serre's argument, making the slight modifications needed to prove
the following generalisation (which is well known to experts).
We refer to the seminal paper of Bruhat and Tits for
the theory of affine buildings \cite{BT}; see also Chapter 9 of \cite{ronan} and Chapter VI of \cite{brown2}.

\begin{prop}\label{p:serre}
If $G$ is a finitely generated group that has a fixed point whenever it acts by semisimple
isometries on a complete \cat space of dimension less than $n$,
and  $k$ is a field, then for every representation $\rho: G\to {\rm{GL}}(n, k)$
and every $g\in G$, the eigenvalues of $\rho(g)$ are integral over $\Z$.
\end{prop}

\begin{proof} The subfield $k_\rho\subset k$  generated by the entries of $\rho(G)$ is finitely
generated over the prime field $\mathbb Q$ or $\mathbb F_p$. Let $v$ be a discrete valuation
on $k_\rho$ --- i.e.~ an epimorphism $k_\rho^*\to \Z$ such that $v(xy)=v(x)v(y)$
and $v(x+y)\ge \min\{v(x),v(y)\}$
for all $x,y\in k_\rho$, with the convention that $v(0)=+\infty$.
The corresponding valuation ring is $\mathcal O_v=\{x\in k_\rho \mid v(x)\ge 0\}$. 

$G$ is perfect, since it can't act freely on a line by isometries, and hence it lies in the
kernel  of $v\circ {\rm{det}}: {\rm{GL}}(n, k_\rho)\to\Z$. This kernel acts by isometries
on the geometric realisation of a Bruhat-Tits building of affine type; in this action each element that fixes a point
fixes a simplex pointwise.
This building is a complete \cat\ space of dimension $n-1$,
so $\rho(G)$ fixes a simplex. Hence $\rho(G)$ is conjugate in ${\rm{GL}}(n, k_\rho)$ to a subgroup of
${\rm{GL}}(n, \mathcal O_v)$, since each vertex stabilizer is. 
It follows that for each $g\in G$ the coefficients of the 
characteristic polynomial of $\rho(g)$ lie in the intersection $\bigcap_v\mathcal O_v$.
This intersection is precisely the set of elements of $k_\rho$ that are integral over $\Z$ (see \cite{ega},
p.140,  corollary 7.1.8). Hence, by the transitivity of integrality, the eigenvalues of $\rho(g)$ are integral
over $\Z$.
\end{proof}

Bass, \cite{bass} Proposition 5.3,
proves that  a finitely generated  group satisfying the conclusion of Proposition \ref{p:serre} has  only finitely many
conjugacy classes of irreducible representations $G\to {\rm{GL}}(n,K)$ for an arbitrary algebraically closed
field $K$.

\end{document}